\documentclass{article}
\usepackage{mathrsfs, amsmath, amsthm, amssymb, bm, mathtools, thm-restate}
\usepackage{graphicx, xcolor, tikz}
\usetikzlibrary{arrows.meta, intersections, calc, shapes}
\usepackage{caption, subcaption}
\usepackage{array}
\usepackage{cases}
\makeatletter
\def\th@plain{%
  \upshape 
}
\makeatother

\makeatletter
\renewenvironment{proof}[1][\proofname]{\par
  \pushQED{\qed}%
  \normalfont \topsep6\p@\@plus6\p@\relax
  \trivlist
  \item[\hskip\labelsep
        \bfseries
    #1\@addpunct{.}]\ignorespaces
}{%
  \popQED\endtrivlist\@endpefalse
}
\makeatother

\newtheorem{theorem}{Theorem}[section]

\newtheorem{lemma}{Lemma}[section]

\newtheorem*{conjecture*}{Conjecture}

\usepackage[left=25mm, top=25mm, bottom=25mm, right=25mm]{geometry}
\usepackage[T1]{fontenc}
\usepackage[inline]{enumitem}
\usepackage[square, numbers, sort&compress]{natbib}

\usepackage[pdftex,%
  bookmarks=true,%
  bookmarksnumbered=true, 
  bookmarksopen=true, 
  plainpages=false,%
  pdfpagelabels,%
  colorlinks=true, 
  linkcolor=black, 
  citecolor=black,%
  anchorcolor=green,
  urlcolor= blue,
  breaklinks=true,
  hyperindex=true]{hyperref}

\newcommand{\ie}{i.e.,\ }
\def\Int(#1){\mathrm{Int}(#1)}
\def\Ext(#1){\mathrm{Ext}(#1)}

\begin{document}
\title{An extension of Thomassen's result on choosability}
\author{Lingxi Li { } \quad Tao Wang\footnote{\tt wangtao@henu.edu.cn; https://orcid.org/0000-0001-9732-1617}\\
{\small Center for Applied Mathematics}\\
{\small Henan University, Kaifeng, 475004, P. R. China}}
\date{}
\maketitle
\begin{abstract}
Thomassen proved that all planar graphs are $5$-choosable. \v{S}krekovski strengthened the result by showing that all $K_{5}$-minor-free graphs are $5$-choosable. Dvo\v{r}\'{a}k and Postle pointed out that all planar graphs are DP-$5$-colorable. In this note, we first improve these results by showing that every $K_{5}$-minor-free or $K_{3, 3}$-minor-free graph is DP-$5$-colorable. In the final section, we further improve these results under the term strictly $f$-degenerate transversal. 
\end{abstract}

\section{Introduction}
Thomassen \cite{MR1290638} proved that all planar graphs are $5$-choosable. \v{S}krekovski \cite{MR1639710} (see also \cite{MR2418108, MR2746713}) extended the result to the class of $K_{5}$-minor-free graphs. Dvo\v{r}\'{a}k and Postle \cite{MR3758240} gave a generalization of list coloring, under the name correspondence coloring, which was called DP-coloring by Bernshteyn, Kostochka, and Pron \cite{MR3686937}. 

Let $G$ be a graph and $L$ be a list assignment for $G$. For each vertex $v \in V(G)$, we associate it with a set $L_{v} = \{v\} \times L(v)$; for each edge $uv \in E(G)$, we associate it with a matching $\mathscr{M}_{uv}$ between $L_{u}$ and $L_{v}$. Let $\mathscr{M} = \bigcup_{uv \in E(G)}\mathscr{M}_{uv}$, and we call $\mathscr{M}$ the {\bf matching assignment} over $L$. The matching assignment $\mathscr{M}$ is called a {\bf $k$-matching assignment} if $L(v) = \{1, 2, \dots, k\}$ for every $v \in V(G)$. A {\bf cover} of $G$ is a graph $H_{L, \mathscr{M}}$ (simply write $H$) meeting two conditions: 
\begin{itemize}
\item the vertex set of $H$ is the disjoint union of $L_{v}$ for all $v \in V(G)$; and
\item the edge set of $H$ is the matching assignment $\mathscr{M}$.
\end{itemize}

Let $G$ be a graph and $H$ be a cover of $G$ over a list assignment $L$. An {\bf $(L, \mathscr{M})$-coloring} of $G$ is an independent set $\mathcal{I}$ of $H$ such that $|\mathcal{I} \cap L_{v}| = 1$ for each $v \in V(G)$. A graph $G$ is {\bf DP-$k$-colorable} if for any list assignment $L(v) \supseteq \{1, 2, \dots, k\}$ and any matching assignment $\mathscr{M}$, it admits an $(L, \mathscr{M})$-coloring. Note that every DP-$k$-colorable graph is $k$-choosable. 

Dvo\v{r}\'{a}k and Postle \cite{MR3758240} have pointed out that all planar graphs are DP-$5$-colorable. We improve the result to the following \autoref{KFIVE}, and we also extend the result for planar graphs to the class of $K_{3, 3}$-minor-free graphs. 
\begin{restatable}{theorem}{KFIVE}\label{KFIVE}
All $K_{5}$-minor-free graphs are DP-$5$-colorable. 
\end{restatable}

\begin{restatable}{theorem}{KTHREE}\label{KTHREE}
All $K_{3, 3}$-minor-free graphs are DP-$5$-colorable. 
\end{restatable}

Let $H$ be a cover of $G$, and let $f$ be a function from $V(H)$ to $\{0, 1, 2, \dots\}$. A subset $T \subseteq V(H)$ is called a {\bf transversal} if $|T \cap L_{v}| = 1$ for each $v \in V(G)$. A transversal $T$ of a cover $H$ is {\bf strictly $f$-degenerate} if every nonempty subgraph $\Gamma$ in $H[T]$ contains a vertex $x$ with $\deg_{\Gamma}(x) < f(x)$. In other words, all the vertices of $H[T]$ can be ordered as $x_{1}, x_{2}, \dots, x_{n}$ such that each vertex $x_{i}$ has less than $f(x_{i})$ neighbors on the right hand side. Such an order is an {\bf $f$-removing order}, and the reverse order $x_{n}, x_{n-1}, \dots, x_{1}$ is an {\bf $f$-coloring order}.

By definition, a vertex $x$ can never be chosen in a strictly $f$-degenerate transversal if $f(x) = 0$. Hence, we can add some vertices into $L_{v}$ and define the value of $f$ to be zero on these new vertices, so that all the $L_{v}$ have the same cardinality. On the other hand, it doesn't matter what the labels of the vertices are, so we may assume that $L_{v} = \{v\} \times [s]$, where $s$ is an integer. A cover $H$ together with a function $f$ is called a valued-cover. 

In \autoref{Sec:3}, we strengthen Theorems \ref{KFIVE} and \ref{KTHREE} to \autoref{F-MINOR}. In order to demonstrate how Thomassen's technique in \cite{MR1290638} is extended, we first give a proof for \autoref{KFIVE} in \autoref{Sec:2}, and then give one for \autoref{F-MINOR}, even though Theorems \ref{KFIVE} and \ref{KTHREE} are special cases of \autoref{F-MINOR}. For a function $f$, we use $R_{f}$ to denote the range of $f$. 

\begin{theorem}\label{F-MINOR}
Assume that $G$ is a $K_{5}$-minor-free or $K_{3, 3}$-minor-free graph, and $(H, f)$ is a valued-cover with $R_{f} \subseteq \{0, 1, 2\}$. Then $H$ contains a strictly $f$-degenerate transversal.  
\end{theorem}

Assume that $G$ is a plane graph and $C$ is a cycle in it. We will use $\Int(C)$ (resp. $\Ext(C)$) to denote the subgraph induced by $V(C)$ and the vertices inside (resp. outside) of $C$. The cycle $C$ is a {\bf separating} cycle of $G$ if both the interior and the exterior of $C$ have at least one vertex. 
\section{DP-5-coloring}\label{Sec:2}
A {\bf plane triangulation} is an embedded plane graph such that each of its faces is bounded by a cycle of length three. A {\bf near-triangulation} is an embedded plane graph such that each bounded face is bounded by a triangle and the unbounded face (outer face) is bounded by a cycle. An {\bf $\bm{\ell}$-sum} of two graphs $G'$ and $G''$ is the graph $G$ such that $G = G' \cup G''$ and $G' \cap G'' = K_{\ell}$. 

The {\bf Wagner graph} is a 3-regular graph with 8 vertices and 12 edges, see \autoref{WagnerFig}. Note that the Wagner graph is non-planar, thus the Wagner graph cannot be a subgraph of a planar graph. 
\begin{figure}[htbp!]
\centering
\begin{tikzpicture}
\def\s{0.707}
\foreach \ang in {1, 2, 3, 4, 5, 6, 7, 8}
{
\def\pointname{v\ang}
\coordinate (\pointname) at ($(\ang*45:\s)$);
}
\draw (v1)--(v2)--(v3)--(v4)--(v5)--(v6)--(v7)--(v8)--cycle;
\foreach \ang in {1, 2, 3, 4, 5, 6, 7, 8}
{
\node[circle, inner sep = 1, fill, draw] () at (v\ang) {};
}
\draw (v1)--(v5);
\draw (v2)--(v6);
\draw (v3)--(v7);
\draw (v4)--(v8);
\end{tikzpicture}
\caption{Wagner graph.}
\label{WagnerFig}
\end{figure}
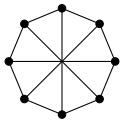

Wagner \cite{MR1513158} gave the following characterization of planar graphs in terms of graph minors. 
\begin{theorem}[Wagner \cite{MR1513158}]
A graph is planar if and only if it does not contain $K_{5}$  or $K_{3, 3}$ as a minor. 
\end{theorem}
By Wagner's Theorem, the class of $K_{5}$-minor-free graphs and the class of $K_{3, 3}$-minor-free graphs are two superclasses of planar graphs. 

A graph $G$ is {\bf maximal $K_{5}$-minor-free} if it does not contain $K_{5}$ as a minor, but $G + xy$ contains a $K_{5}$-minor for every pair nonadjacent vertices $x$ and $y$ in $G$. Wagner \cite{MR1513158} also gave the following characterization of maximal $K_{5}$-minor-free graphs. 
\begin{theorem}[Wagner \cite{MR1513158}]\label{23SUMS}
Every maximal $K_{5}$-minor-free graph can be obtained from the Wagner graph and plane triangulations by recursively $2$-sums or $3$-sums. 
\end{theorem}

The following theorem and its proof are very similar to that in \cite{MR1290638}, but for completeness we give a complete proof here. 
\begin{theorem}\label{NT}
Assume that $G$ is a near-triangulation such that the outer face is bounded by a cycle $\mathcal{O}  = v_{1}v_{2}\dots v_{p}v_{1}$. Let $L$ be a list assignment of $G$ such that $|L(v)| \geq 3$ for each $v \in V(\mathcal{O})$ and $|L(v)| \geq 5$ for each $v \notin V(\mathcal{O})$. If $\mathscr{M}$ is a matching assignment for $G$ and $R_{0}$ is an $(L, \mathscr{M})$-coloring of $G[\{v_{1}, v_{2}\}]$, then $G$ admits an $(L, \mathscr{M})$-coloring such that its restriction on $G[\{v_{1}, v_{2}\}]$ is $R_{0}$. 
\end{theorem}
\begin{proof}
The assertion is proved by induction on $|V(G)|$. When $G$ has only three vertices, $G = \mathcal{O} = K_{3}$ and the assertion is obvious. So we can assume that $|V(G)| \geq 4$ and the assertion is true for smaller graphs. Suppose that $\mathcal{O}$ has a chord $v_{i}v_{j}$. It follows that $v_{i}v_{j}$ lies in two cycles $C_{1}$ and $C_{2}$ of $\mathcal{O} + v_{i}v_{j}$. Let $v_{1}v_{2}$ lie in $C_{1}$. Applying the induction hypothesis to $\Int(C_{1})$, $R_{0}$ can be extended to an $(L, \mathscr{M})$-coloring of $\Int(C_{1})$. After $v_{i}$ and $v_{j}$ are colored, it can be further extended to an $(L, \mathscr{M})$-coloring of $\Int(C_{2})$. This  yields a desired $(L, \mathscr{M})$-coloring of $G$. 

So we can assume that $\mathcal{O}$ has no chord. Let $v_{1}, u_{1}, u_{2}, \dots, u_{m}, v_{p-1}$ be the neighbors of $v_{p}$ in a natural cyclic order around $v_{p}$. Since all the bounded faces of $G$ are bounded by triangles and $\mathcal{O}$ has no chord, $P = v_{1}u_{1}u_{2} \dots u_{m}v_{p-1}$ is a path and $\mathcal{O}' = P \cup (\mathcal{O} - v_{p})$ is a cycle. Let $j$ and $\ell$ be two distinct elements in $L(v_{p})$ which do not conflict with the color of $v_{1}$ under the matching $\mathscr{M}_{v_{1}v_{p}}$. Now define $L'(v) = L(v)$ for every $v \notin \{u_{1}, u_{2}, \dots, u_{m}, v_{p}\}$, for $1 \leq i \leq m$, define $L'(u_{i})$ from $L(u_{i})$ by deleting the neighbors of $j, \ell \in L(v_{p})$ under the matching $\mathscr{M}_{v_{p}u_{i}}$. It is easy to check that $|L'(v)| \geq 3$ for all $v \in V(\mathcal{O}')$ and $|L'(v)| \geq 5$ for all $V(G) - \{v_{p}\} - V(\mathcal{O}')$. Applying the induction hypothesis to $\mathcal{O}'$ and its interior and the new list $L'$, we have an $(L', \mathscr{M})$-coloring for $G - v_{p}$. There is at least one color in $\{j, \ell\} \subset L(v_{p})$ which do not conflict with the color of $v_{p-1}$ under $\mathscr{M}_{v_{p-1}v_{p}}$, so we can assign it to the vertex $v_{p}$. This completes the proof. 
\end{proof}

\begin{theorem}\label{K5-FREE}
Assume that $G$ is a maximal $K_{5}$-minor-free graph. If $K$ is a subgraph of $G$ isomorphic to $K_{2}$ or $K_{3}$, then every DP-$5$-coloring $\varphi$ of $K$ can be extended to a DP-$5$-coloring of $G$. 
\end{theorem}
\begin{proof}
Suppose to the contrary that $G$ is a counterexample with $|V(G)|$ as small as possible. 

Assume that $G$ is a plane triangulation and $K$ is a separating $3$-cycle of $G$. Note that $\Int(K)$ and $\Ext(K)$ are both plane triangulations and maximal $K_{5}$-minor-free graphs. By minimality, every DP-$5$-coloring $\varphi$ of $K$ can be extended to a DP-$5$-coloring $\varphi_{1}$ of $\Int(K)$ and a DP-$5$-coloring $\varphi_{2}$ of $\Ext(K)$. Combining $\varphi_{1}$ and $\varphi_{2}$ yields a DP-$5$-coloring of $G$, a contradiction. 

Assume that $G$ is a plane triangulation and $K = [x_{1}x_{2}x_{3}]$ bounds a $3$-face. Note that $G$ has at least four vertices. We can redraw the plane triangulation such that $K$ is the boundary of the outer face. Note that $G - x_{3}$ is a near-triangulation. Since $x_{3}$ on $K$ is precolored, every uncolored vertex incident with the outer face of $G - x_{3}$ has at least four admissible colors other than $\varphi(x_{3})$. Applying \autoref{NT} to $G - x_{3}$, we obtain a DP-$5$-coloring of $G$ whose restriction on $K$ is the precoloring $\varphi$. 

Assume that $G$ is a plane triangulation and $K = y_{1}y_{2}$. We can further assume that $y_{1}y_{2}$ is incident with a $3$-face $[y_{1}y_{2}y_{3}]$. Clearly, the precoloring of $K$ can be extended to a DP-$5$-coloring of $G[y_{1}, y_{2}, y_{3}]$, and we can reduce the problem to the previous case. 

If $G$ is the Wagner graph, then we can greedily extend the precoloring of $K$ to a DP-$5$-coloring of $G$ since $G$ is $3$-regular. 

By \autoref{23SUMS}, we can assume that $G$ is a $2$-sum or $3$-sum of two maximal $K_{5}$-minor-free graphs $G_{1}$ and $G_{2}$ with $K \subset G_{1}$. By minimality, the precoloring $\varphi$ of $K$ can be extended to a DP-$5$-coloring $\varphi_{1}$ of $G_{1}$. By minimality once again, we can extended the restriction of $\varphi_{1}$ on $G_{1} \cap G_{2}$ to $G_{2}$. This yields a DP-$5$-coloring of $G$ whose restriction on $K$ is the precoloring $\varphi$.   
\end{proof}

Now, we can easily prove \autoref{KFIVE}. 
\KFIVE*
\begin{proof}
Since every $K_{5}$-minor-free graph is a spanning subgraph of a maximal $K_{5}$-minor-free graph, it suffices to prove the result for maximal $K_{5}$-minor-free graphs. We can first color two adjacent vertices in $G$, and extend the coloring to the whole graph according to \autoref{K5-FREE}.
\end{proof}

Wagner \cite{MR1513158} also gave a characterization of maximal $K_{3, 3}$-minor-free graphs by $2$-sums. 
\begin{theorem}[Wagner \cite{MR1513158}]
Every maximal $K_{3, 3}$-minor-free graph can be obtained from the complete graph $K_{5}$ and plane triangulations by recursively $2$-sums. 
\end{theorem}
Since the proof of the following result is similar to that in \autoref{K5-FREE}, we leave it as an exercise to the readers. 
\begin{theorem}\label{K33-FREE}
Assume that $G$ is a maximal $K_{3, 3}$-minor-free graph. If $K$ is a subgraph of $G$ isomorphic to $K_{2}$, then every DP-$5$-coloring of $K$ can be extended to a DP-$5$-coloring of $G$. 
\end{theorem}

\KTHREE*
\begin{proof}
Since each $K_{3, 3}$-minor-free graph is a spanning subgraph of a maximal $K_{3, 3}$-minor-free graph, it suffices to show the result for maximal $K_{3, 3}$-minor-free graphs. We can first color two adjacent vertices in $G$, and further extend the precoloring to the whole graph according to \autoref{K33-FREE}. 
\end{proof}

\section{Strictly $f$-degnerate transversal}\label{Sec:3}
In this section, we extend the results on DP-$5$-coloring to particular strictly $f$-degenerate transversal. The following two lemmas were presented by Nakprasit and Nakprasit \cite[Lemma 2.3]{MR4114324} with a different term. 

For a vertex subset $K$ of $V(G)$, or a subgraph $K$ of $G$, we use $H_{K}$ to denote the cover restricted on $K$, \ie $H_{K} \coloneqq H[\bigcup_{v \in K} L_{v}]$. 

\begin{lemma}\label{ORDER}
Assume that $G$ is a graph and $K$ is a subgraph of $G$. Let $(H, f)$ be a valued cover, and $T$ be a transversal of $H_{K}$ such that $H[T]$ has no edges and $f(x) = 1$ for each $x \in T$. If $T$ can be extended to a strictly $f$-degenerate transversal $T'$ of $H$, then there exists an $f$-removing order of $T'$ such that the vertices in $T$ are on the rightest of the order. 
\end{lemma}
\begin{proof}
Let $S'$ be an $f$-removing order of $T'$. Since $f(x) = 1$ for each $x \in T$, every vertex in $T$ has no neighbor on the right of the order $S'$, so we can move all the vertices in $T$ to the rightest of the order. In other words, we can delete all the vertices in $T$ from the order $S'$ and put the vertices in $T$ on the right side of all the other vertices of $S'$. Observe that the resulting order satisfies the desired condition. 
\end{proof}

\begin{lemma}\label{F-EXTEND}
Assume that $G = G_{1} \cup G_{2}$, $V(G_{1} \cap G_{2}) = K$ and $G_{1}$ is an induced subgraph of $G$. Let $(H, f)$ be a valued cover of $G$, and $H_{i}$ be the restriction of $H$ on $G_{i}$ for $i \in \{1, 2\}$. If $R$ is a strictly $f$-degenerate transversal of $H_{1}$, and $R \cap H_{K}$ can be extended to a strictly $f^{*}$-degenerate transversal $R^{*}$ of $H^{*}$, where $H^{*}$ is obtained from $H_{2}$ by deleting all the edges in $H_{K}$, and $f^{*}$ is obtained from $f$ by defining $f^{*}(x) = 1$ for each $x \in R \cap H_{K}$, then $R \cup R^{*}$ must be a strictly $f$-degenerate transversal of $H$. 
\end{lemma}
\begin{proof}
It suffices to give an $f$-removing order of $H[R \cup R^{*}]$. By \autoref{ORDER}, there exists an $f^{*}$-removing order of $R^{*}$ such that the vertices in $R \cap H_{K}$ are on the rightest of the order. Then we list all the vertices of $R^{*} \setminus (R \cap H_{K})$ according to the $f^{*}$-removing order and then list the vertices of $R$ according to an $f$-removing order. It is easy to check that the resulting order is an $f$-removing order for $H[R \cup R^{*}]$. 
\end{proof}

We first extend \autoref{NT} to the following result. Note that \autoref{F-NT} was first proved in \cite[Theorem 1.6]{MR4114324}, but the following proof is a little bit different from that one. 

\begin{theorem}\label{F-NT}
Assume that $G$ is a near-triangulation such that the outer face is bounded by a cycle $\mathcal{O}  = v_{1}v_{2}\dots v_{p}v_{1}$. Let $(H, f)$ be a valued cover of $G$ with $R_{f} \subseteq \{0, 1, 2\}$ such that 
\begin{equation}
f(v, 1) + \dots + f(v, s) \geq 3 \mbox{ for every $v \in V(\mathcal{O})$}
\end{equation}
and 
\begin{equation}
f(v, 1) + \dots + f(v, s) \geq 5 \mbox{ for every $v \notin V(\mathcal{O})$}. 
\end{equation}
If $R_{0}$ is a strictly $f$-degenerate transversal of $H[L_{v_{1}} \cup L_{v_{2}}]$, then $R_{0}$ can be extended to a strictly $f$-degenerate transversal of $H$. 
\end{theorem}
\begin{proof}
We prove the assertion by induction on $|V(G)|$. When $G$ has exactly three vertices, $G = \mathcal{O} = K_{3}$ and the assertion is obvious. Then $|V(G)| \geq 4$ and the assertion is true for smaller graphs. Suppose that $\mathcal{O}$ has a chord $uw$. It follows that $uw$ lies in two cycles $C_{1}$ and $C_{2}$ of $\mathcal{O} + uw$ with $v_{1}v_{2}$ in $C_{1}$. Let $G_{1} : = \Int(C_{1})$ and $G_{2} : = \Int(C_{2})$. Applying the induction hypothesis to $G_{1}$, $R_{0}$ can be extended to a strictly $f$-degenerate transversal $R$ of $H_{1}$, and then $R \cap H[L_{u} \cup L_{w}]$ can be extended to a strictly $f^{*}$-degenerate transversal $R^{*}$ of $H^{*}$ as in \autoref{F-EXTEND}. Therefore, $R^{*} \cup R$ is a desired strictly $f$-degenerate transversal of $H$. 

The other case is that $\mathcal{O}$ has no chord. Let $v_{1}, u_{1}, u_{2}, \dots, u_{m}, v_{p-1}$ be the neighbors of $v_{p}$ in a natural cyclic order around $v_{p}$, and let $U = \{u_{1}, u_{2}, \dots, u_{m}\}$. Since all the bounded faces of $G$ are bounded by triangles and $\mathcal{O}$ has no chord, we have $P = v_{1}u_{1}u_{2} \dots u_{m}v_{p-1}$ is a path and $\mathcal{O}' = P \cup (\mathcal{O} - v_{p})$ is a cycle. For each $x \in \{v_{p}\} \times [s]$, let 
\[
f'(x) = 
\begin{cases}
\max\{0, f(x) - 1\}, & \text{if $x$ is adjacent to $R_{0} \cap L_{v_{1}}$ under $\mathscr{M}_{v_{1}v_{p}}$}; \\[0.5cm]
f(x), & \text{otherwise.}
\end{cases}
\]
Since $R_{0} \cap L_{v_{1}}$ has at most one neighbor in $L_{v_{p}}$, we have $f'(v_{p}, 1) +  \dots + f'(v_{p}, s) \geq 2$. Let 
\[
X' = \{\,x \in \{v_{p}\} \times [s] : f'(x) > 0\,\}.
\]

{\bf Case 1. $\bm{|X'| \geq 2}$}. 

Let $X^{*}$ be a subset of $X'$ with $|X^{*}| = 2$. A new function $f^{
\dag}$ on $H - L_{v_{p}}$ is defined as 
\[
f^{\dag}(x) = 
\begin{cases}
\max\{0, f(x) - 1\}, & \text{if $x \in U \times [s]$ and $x$ is connected to a vertex in $X^{*}$}; \\[0.5cm]
f(x), & \text{otherwise.}
\end{cases}
\]
It follows that, for each $u \in \mathcal{O}'$, we have 
\[
\sum_{z \in L_{u}} f^{\dag}(z) \geq 3.
\] 

By induction hypothesis and \autoref{ORDER}, $(H - L_{v_{p}}, f^{\dag})$ contains a strictly $f^{\dag}$-degenerate transversal $R^{\dag}$ with an $f^{\dag}$-removing order $S^{\dag}$ such that the vertices in $R_{0}$ are on the rightest of the order. Let $(v_{p}, c_{p})$ be a vertex in $X^{*}$ which is not adjacent to $R^{\dag} \cap L_{v_{p-1}}$. Therefore, we insert $(v_{p}, c_{p})$ into $S^{\dag}$ such that it is the reciprocal third element to obtain an $f$-removing order of a strictly $f$-degenerate transversal of $H$. 

{\bf Case 2. $\bm{|X'| = 1}$}.

Without loss of generality, assume that $X' = \{(v_{p}, 1)\}$. Since $f'(v_{p}, 1) + \dots + f'(v_{p}, s) \geq 2$ and $R_{f} \subseteq \{0, 1, 2\}$, we have $f'(v_{p}, 1) = 2$. Define a function $f^{
\dag}$ on $H - L_{v_{p}}$ by
\[
f^{\dag}(x) = 
\begin{cases}
0, & \text{if $x \in U \times [s]$ and $x$ is adjacent to $(v_{p}, 1)$ in $H$}; \\[0.5cm]
f(x), & \text{otherwise.}
\end{cases}
\]
Note that the range of $f$ is a subset of $\{0, 1, 2\}$, for each $u \in \mathcal{O}'$,  
\[
\sum_{z \in L_{u}} f^{\dag}(z) \geq 3.
\]
By induction hypothesis, $(H - L_{v_{p}}, f^{\dag})$ admits a strictly $f^{\dag}$-degenerate transversal $R^{\dag}$ with an $f^{\dag}$-removing order $S^{\dag}$ such that the vertices in $R_{0}$ are on the rightest of the order. Let $S$ be a sequence obtained from $S^{\dag}$ by inserting $(v_{p}, 1)$ into $S^{\dag}$ such that $(v_{p}, 1)$ is the immediate predecessor of $(v_{p-1}, c_{p-1})$, where $(v_{p-1}, c_{p-1}) \in L_{v_{p-1}} \cap R^{\dag}$. Recall that $f^{\dag}(v_{p}, 1) = 2$, it is not hard to check that $S$ is an $f$-removing order of a strictly $f$-degenerate transversal of $H$. 
\end{proof}

Instead of proving \autoref{F-MINOR}, we prove the following stronger theorem for $K_{5}$-minor-free graphs, and leave the corresponding result for $K_{3, 3}$-minor-free graphs to the readers. 

\begin{theorem}\label{F-M}
Assume that $G$ is a $K_{5}$-minor-free graph, and $(H, f)$ is a valued-cover with $R_{f} \subseteq \{0, 1, 2\}$. If $K$ is a subgraph isomorphic to $K_{2}$ or $K_{3}$, and $f(v, 1) + \dots + f(v, s) \geq 5$ for each $v \in V(G)$, then every strictly $f$-degenerate transversal of $H_{K}$ can be extended to a strictly $f$-degenerate transversal of $H$. 
\end{theorem}
\begin{proof}
Suppose to the contrary that $(G, H, f, R_{0})$ is a counterexample with $|V(G)|$ as small as possible, where $R_{0}$ is a strictly $f$-degenerate transversal of $H_{K}$. Similar to the previous results, we only need to consider the case that $G$ is a maximal $K_{5}$-minor-free graph. 

Assume that $G$ is a plane triangulation and $K$ is a separating triangle of $G$. Note that $\Ext(K)$ and $\Int(K)$ are both plane triangulations and maximal $K_{5}$-minor-free graphs. By minimality and \autoref{F-EXTEND}, $R_{0}$ can be extended to a strictly $f$-degenerate transversal of $H$.

Assume that $G$ is a plane triangulation and $K = [x_{1}x_{2}x_{3}]$ bounds a $3$-face. We can redraw the plane triangulation such that $K$ bounds the outer face. Let $(x_{3}, c_{3})$ be in $R_{0}$, define a function $f'$ on $H - L_{x_{3}}$ by
\[
f'(x) = 
\begin{cases}
0, & \text{if $x \in \{u\} \times [s]$ with $u \notin \{x_{1}, x_{2}\}$ and $x$ is connected to $(x_{3}, c_{3})$ in $H$}; \\[0.5cm]
f(x), & \text{otherwise.}
\end{cases}
\]
Note that the graph $G - x_{3}$ is a near-triangulation. Since the range of $f$ is a subset of $\{0, 1, 2\}$, we have that, for each $w$ on the outer face of $G - x_{3}$, \[\sum_{x \in \{w\} \times [s]} f'(x) \geq 3.\] By \autoref{F-NT}, $R_{0}\setminus \{(x_{3}, c_{3})\}$ can be extended to a strictly $f'$-degenerate transversal of $H \setminus L_{x_{3}}$ with an $f'$-removing order $S'$ such that the two vertices in $R_{0}\setminus \{(x_{3}, c_{3})\}$ are on the rightest of the order. According to an $f$-removing order of $R_{0}$, we can insert $(x_{3}, c_{3})$ into $S'$ such that the three vertices in $R_{0}$ are the three rightest elements in the order to obtain an $f$-removing order of a strictly $f$-degenerate transversal of $H$. 

Assume that $G$ is a plane triangulation and $K = x_{1}x_{2}$. We may assume that $x_{1}x_{2}$ is incident with a $3$-face $[x_{1}x_{2}x_{3}]$. Clearly, $R_{0}$ can be extended to a strictly $f$-degenerate transversal of $H_{[x_{1}, x_{2}, x_{3}]}$, and we can reduce the problem to the previous case. 

If $G$ is the Wagner graph, then we can greedily extend $R_{0}$ to a strictly $f$-degenerate transversal of $H$ since $G$ is 3-regular. 

By \autoref{23SUMS}, assume that $G$ is a $2$-sum or $3$-sum of two maximal $K_{5}$-minor-free graphs $G_{1}$ and $G_{2}$ with $K \subset G_{1}$. By minimality and \autoref{F-EXTEND}, $R_{0}$ can be extended to a strictly $f$-degenerate transversal of $H$. 
\end{proof}

In Theorems \ref{F-NT} and \ref{F-M}, there is a restriction on $f$, \ie the range of $f$ is a subset of $\{0, 1, 2\}$. If the restriction can be dropped, the results can imply two theorems due to Thomassen. Thomassen proved that every planar graph can be partitioned into a $3$-degenerate graph and an independent set \cite{MR1866722}, and every planar graph can be partitioned into a $2$-degenerate graph and a forest \cite{MR1358992}. So the second author and some others made the following conjecture in \cite{MR4357325}. 

\begin{conjecture*}
Assume that $G$ is a planar graph and $(H, f)$ is a positive-valued cover. If $s \geq 2$ and $f(v, 1) +  \dots + f(v, s) \geq 5$ for each $v \in V(G)$, then $H$ admits a strictly $f$-degenerate transversal. 
\end{conjecture*}

\end{document}